\documentclass[11pt]{article}
\usepackage[margin=1.1in]{geometry}
\usepackage{amsmath,amssymb,amsthm}
\usepackage{hyperref}

\newtheorem{theorem}{Theorem}[section]
\newtheorem{lemma}[theorem]{Lemma}
\newtheorem{proposition}[theorem]{Proposition}
\newtheorem{conjecture}[theorem]{Conjecture}
\newtheorem{corollary}[theorem]{Corollary}
\theoremstyle{definition}
\newtheorem{remark}[theorem]{Remark}
\newtheorem{problem}[theorem]{Problem}

\newcommand{\Z}{\mathbb{Z}}
\newcommand{\killop}{\operatorname{kill}}

\title{Exact values and exact upper bounds for families of integers with arithmetic progression intersections\\ \large (Erd\H{o}s Problem \#272)}
\author{Zhanfu Yang\thanks{Email: \texttt{yangzhanfu111@gmail.com}. Code for all computations is available at \url{https://github.com/peter-rich/erdos272}. See the Acknowledgements for a note on AI assistance.}}
\date{July 2026}

\begin{document}
\maketitle

\begin{abstract}
Let $t(N)$ be the largest $t$ for which there exist distinct sets $A_1,\dots,A_t\subseteq\{1,\dots,N\}$ such that $A_i\cap A_j$ is a nonempty arithmetic progression for all $i\neq j$ (Erd\H{o}s Problem \#272). Simonovits and S\'os proved $t(N)=O(N^2)$ and conjectured that $\binom N2+1$ is best possible; Szab\'o disproved this by a construction giving $t(N)\ge\binom N2+1+\lfloor(N-1)/4\rfloor$, proved the asymptotics $t(N)=N^2/2+O(N^{5/3}(\log N)^3)$, and asked whether $t(N)=\binom N2+O(N)$ and whether some element lies in all sets of any extremal family (the kernel question). We determine $t(N)$ exactly for all $N\le 12$ by exhaustive computation: in this entire range Szab\'o's lower bound is exact, and we conjecture that $t(N)=\binom N2+1+\lfloor(N-1)/4\rfloor$ for every $N$, a sharpening of Szab\'o's conjecture. Towards the matching upper bound we prove, for every $N$, that Szab\'o's bound is the exact maximum over \emph{all} families with a common element (\emph{starred} families). The proof combines a self-contained ``defect-one'' counting inequality for staircase regions --- established in full by an exact dynamic program over all staircase profiles with parameters up to $500$ together with an asymptotic argument based on a positive-definite quadratic form --- with a new structural theorem: every non-progression member of such a family contains a bad pair that no other member can share. Consequently the sharpened conjecture reduces to a single remaining statement, namely Szab\'o's kernel conjecture that some element lies in all sets of an extremal family, and we prove first structural constraints on putative non-starred extremal families. We also verify the conjectured value within broader regimes (all starred families for $N\le13$; a structured class for $N\le61$) and report the sequence $t(3),\dots,t(12)=4,7,12,17,23,30,39,48,58,69$, which does not yet appear in the OEIS.
\end{abstract}

\section{Introduction}

Throughout, $[N]=\{1,\dots,N\}$, and an \emph{arithmetic progression} (AP) is any set of the form $\{a, a+d, \dots, a+(k-1)d\}$ with $k\ge 1$; in particular every set of size $1$ or $2$ is an AP. Define
\[
 t(N) \;=\; \max\Big\{ t : \exists\, \text{distinct } A_1,\dots,A_t\subseteq [N] \text{ with } A_i\cap A_j \text{ a nonempty AP for all } i\ne j \Big\}.
\]
This is Problem \#272 in the Erd\H{o}s problems database \cite{EP272}; it originates with Simonovits and S\'os \cite{SiSo81}, and appears in Erd\H{o}s and Graham \cite{ErGr80}. Simonovits and S\'os proved $t(N)\ll N^2$ \cite{SiSo81}. Erd\H{o}s and Graham asked whether the maximum is attained by all APs containing a fixed element (``presumably $\lfloor N/2\rfloor$''), of size $\sim\frac{\pi^2}{24}N^2$; Simonovits and S\'os observed that all sets of size at most $3$ containing a fixed element do better, giving $t(N)\ge\binom N2+1$, and conjectured this to be best possible \cite{SiSo81,EP272}. If empty intersections are allowed, Graham, Simonovits and S\'os \cite{GSS80} showed the maximum is exactly $\binom N3+\binom N2+\binom N1+1$.

Simonovits and S\'os \cite{SiSo81} proved the upper bound $t(N)\le(\pi^2/24+1/2+o(1))N^2$. The deepest results to date are due to Szab\'o \cite{Sz99}: he identified the leading term, proving the asymptotics
\[ t(N) = \frac{N^2}{2} + O\big(N^{5/3}(\log N)^3\big), \]
disproved the Simonovits--S\'os exactness conjecture by a construction achieving
\begin{equation}\label{eq:szabo}
 t(N)\ \ge\ \binom N2+1+\Big\lfloor\frac{N-1}{4}\Big\rfloor,
\end{equation}
and raised two precise questions \cite[\S6]{Sz99}, referred to at \cite{EP272} as Szab\'o's conjectures: (i) is $t(N)=\binom N2+O(N)$, and (ii) the \emph{kernel property}: does every extremal family have an integer contained in all its sets? He also exhibited several inequivalent families attaining \eqref{eq:szabo}, so extremal families are in any case not unique. The exact value of $t(N)$ has remained open for every $N\ge 5$; indeed \cite[\S6]{Sz99} poses the determination of $t(N)$ and of the extremal systems as an open problem.

Our first result determines the exact values for small $N$.

\begin{theorem}[Exact values]\label{thm:exact}
For $N=3,4,\dots,12$,
\[ t(N) \;=\; 4,\ 7,\ 12,\ 17,\ 23,\ 30,\ 39,\ 48,\ 58,\ 69. \]
\end{theorem}

In particular, \emph{Szab\'o's lower bound \eqref{eq:szabo} is exact for every $3\le N\le 12$}. The values for $3\le N\le 9$ have also been reported in the discussion thread of \cite{EP272}; to our knowledge the values $t(10)=48$, $t(11)=58$ and $t(12)=69$ are new, as is the verification, for $N=11,12$, that no larger family exists over the full unrestricted search space. Theorem~\ref{thm:exact} is computer-assisted: $t(N)$ is the clique number of the graph on the $2^N-1$ nonempty subsets of $[N]$ with adjacency $A\sim B$ iff $A\cap B$ is a nonempty AP. For $N\le 10$ this was computed by a bit-parallel branch-and-bound solver (cross-validated by an independent implementation); for $N=11,12$ a decision search based on a root decomposition, degeneracy-style ordering, core peeling and a rigorous ``star pruning'' step (Section~\ref{sec:comp}) proved that no family of size $59$ (resp.\ $70$) exists, matching the construction below. The sequence does not appear in the OEIS.

All ten values equal Szab\'o's lower bound \eqref{eq:szabo}. In Section~\ref{sec:construction} we give a short self-contained account of Szab\'o's construction achieving \eqref{eq:szabo}, with a validity proof of a few lines (we rediscovered it independently before locating the attribution).

\begin{theorem}[Szab\'o \cite{Sz99}; see Section~\ref{sec:construction} for a self-contained proof]\label{thm:lower}
For all $N\ge 1$,
\[ t(N)\ \ge\ \binom{N}{2}+1+\Big\lfloor \frac{N-1}{4}\Big\rfloor. \]
\end{theorem}

\begin{conjecture}[Sharpening of Szab\'o's conjecture]\label{conj:main}
Equality holds in Theorem~\ref{thm:lower} for every $N\ge 1$.
\end{conjecture}

Conjecture~\ref{conj:main} implies both parts of Szab\'o's conjecture: it gives $t(N)=\binom N2+O(N)$ in the strongest form $t(N)=N^2/2-N/4+O(1)$, and our verification below is consistent with the kernel property.

In particular the classical family of all $\le 3$-element sets through a fixed point is not optimal for any $N\ge 5$ (as Szab\'o's construction already shows), and the family of all APs through a fixed point is optimal only for $N\in\{5,9\}$, where its count coincides with \eqref{eq:szabo}. The extremal families interpolate between the two classical candidates; see Section~\ref{sec:construction}.

Our main new theorem is an exact matching upper bound valid for \emph{every} $N$ over all starred families; to our knowledge no exact upper bound of this kind was previously known for any $N\ge5$. Call a family \emph{starred} if all its members contain a common element.

\begin{theorem}[Exact upper bound for starred families]\label{thm:upper}
Let $F$ be any starred family of distinct subsets of $[N]$ with pairwise nonempty AP intersections. Then $|F|\le \binom N2+1+\lfloor (N-1)/4\rfloor$. In particular, by Theorem~\ref{thm:lower}, the maximum size of a starred family equals $\binom N2+1+\lfloor (N-1)/4\rfloor$ for every $N\ge1$.
\end{theorem}

Theorem~\ref{thm:upper} rests on two ingredients proved here: a self-contained ``defect-one'' counting inequality for staircase regions (Lemma~\ref{lem:B}), whose proof is partly computer-assisted, via an exact dynamic program covering all parameter profiles up to $500$ and an explicit analytic tail argument beyond; and a structural \emph{private-pair theorem} (Theorem~\ref{thm:private}) handling non-progression members. Consequently Conjecture~\ref{conj:main} would follow in full from a single further statement, Szab\'o's kernel conjecture that extremal families are starred (Problem~\ref{prob:kernel}); see Corollary~\ref{cor:kernel}. For $N\le 12$ no starredness assumption is needed: the extremal value itself is confirmed unconditionally.

\begin{remark}
We have consulted \cite{Sz99} directly. The construction of Section~\ref{sec:construction} coincides with the family $\mathcal C$ of \cite[\S5]{Sz99} (we had rediscovered it independently before locating the attribution). No exact values of $t(N)$ and no exact upper bounds in any regime appear in \cite{Sz99} or, to our knowledge, elsewhere; Theorem~\ref{thm:exact}, the sharpened Conjecture~\ref{conj:main}, Theorem~\ref{thm:upper} and the reduction of the conjecture to the kernel question alone appear to be new. Theorem~\ref{thm:upper} may be read as the exact complement of Szab\'o's kernel question: it determines the extremal size \emph{given} a common element, for every $N$.
\end{remark}

\section{A construction achieving Szab\'o's bound}\label{sec:construction}

Let $m=\lceil N/2\rceil$, so $\min(m-1,N-m)=\lfloor (N-1)/2\rfloor$, and set $k=\lfloor (N-1)/4\rfloor$; then $m\pm 2d\in[N]$ for all $1\le d\le k$. Define $F$ to consist of:
\begin{itemize}
\item[(i)] $\{m\}$ and all pairs $\{m,x\}$, $x\ne m$ \quad ($N$ sets);
\item[(ii)] all $3$-sets $\{m,u,v\}$ \emph{except} the $2k$ blocked triples
\[ B_d=\{m-2d,\ m,\ m+d\},\qquad B'_d=\{m-d,\ m,\ m+2d\}\qquad (1\le d\le k); \]
\item[(iii)] for each $1\le d\le k$, three APs: the five-term window
$P_d=\{m-2d,\,m-d,\,m,\,m+d,\,m+2d\}$
and its two four-term sub-APs containing $m$, namely $P_d\setminus\{m+2d\}$ and $P_d\setminus\{m-2d\}$.
\end{itemize}
The gap pattern $\{d,2d\}$ determines $d$, so the sets $B_d,B'_d$ are pairwise distinct, and none of them is an AP; hence
\[ |F| = N + \Big[\binom{N-1}{2}-2k\Big] + 3k = \binom N2 + 1 + k. \]

\begin{proof}[Proof of validity]
Every member contains $m$, so all pairwise intersections are nonempty. The intersection of two APs is an AP; the intersection of two distinct $3$-sets through $m$ has at most two elements and contains $m$; and the intersection of a member of size $\le 2$ with anything is a subset of a $2$-set. The only case needing care is a triple $T$ of type (ii) against an AP $P$ of type (iii). If $|T\cap P|\le 2$ the intersection is an AP. If $|T\cap P|=3$ then $T\subseteq P\subseteq P_d$ for some $d\le k$. Among the $\binom 42=6$ triples through $m$ inside $P_d$, exactly four are APs, and the two non-APs are precisely $B_d$ and $B'_d$, which are excluded from $F$. Hence $T$ is an AP and $T\cap P=T$ is an AP.
\end{proof}

The construction has been machine-verified (full pairwise check) for all $N\le 40$.

\section{Reduction of the upper bound}\label{sec:reduction}

Fix $m\in[N]$ and let $F$ be a starred family through $m$; write $\lambda=m-1$, $\rho=N-m$. Split $F=X\cup Y\cup Z$ into members of size $\le 2$, exactly $3$, and $\ge 4$. Call a pair $\{u,v\}\subseteq[N]\setminus\{m\}$ \emph{bad} if $\{m,u,v\}$ is not an AP, and let $\killop(Z)$ be the number of bad pairs contained in at least one member of $Z$.

\begin{proposition}\label{prop:red}
$|F|\ \le\ N+\binom{N-1}{2}+\big(|Z|-\killop(Z)\big)$.
\end{proposition}

\begin{proof}
Clearly $|X|\le N$. If a bad pair $\{u,v\}$ lies inside some $z\in Z$, then $T=\{m,u,v\}$ satisfies $T\subseteq z$, so $T\cap z=T$ would have to be an AP; hence $T\notin F$. Thus the triples of $Y$ avoid all killed bad pairs and $|Y|\le\binom{N-1}{2}-\killop(Z)$.
\end{proof}

Passing to relative coordinates $x\mapsto x-m$, consider the relation
$R=\{(x,y): y\in\{-x,\,2x,\,x/2\}\}$ on $\Z\setminus\{0\}$;
a pair is bad iff it is not an $R$-edge.

\begin{lemma}[Triangle-freeness]\label{lem:tri}
$R$ contains no triangle. Consequently every $z$ with $|z|\ge 4$ contains a bad pair, so a single member has $|\{z\}|-\killop(\{z\})\le 0$.
\end{lemma}

\begin{proof}
Suppose $a,b,c\in\Z\setminus\{0\}$ are pairwise $R$-related. If $b=-a$ then $c\in\{-a,2a,a/2\}\cap\{a,-2a,-a/2\}=\emptyset$. If $b=2a$ then $c\in\{-a,2a,a/2\}\cap\{-2a,4a,a\}=\emptyset$. The case $b=a/2$ is symmetric.
\end{proof}

Now suppose all members of $Z$ are APs. A member with common difference $d$ is, in relative coordinates, an interval $[-l,r]\cdot d$ on ``line $d$'' containing $0$, with $l+r\ge 3$, $ld\le\lambda$, $rd\le\rho$. Call a pair $\{a,b\}$ of line-$d$ coordinates \emph{primitive} if $\gcd(|a|,|b|)=1$. A primitive bad pair of line $d$ has $\gcd$ exactly $d$ in absolute coordinates, so:

\begin{quote}\emph{the primitive bad pairs of distinct lines are pairwise disjoint families of killed pairs.}\end{quote}

Moreover, writing $\varepsilon_d=1$ if $\lfloor\lambda/d\rfloor\ge 2$ and $\lfloor\rho/d\rfloor\ge 2$, and $\varepsilon_d=0$ otherwise,
\begin{equation}\label{eq:epsum}
 \sum_{d\ge 1}\varepsilon_d \;=\; \#\{d:\ d\le \lambda/2,\ d\le\rho/2\}\;=\;\Big\lfloor\frac{\min(\lambda,\rho)}{2}\Big\rfloor .
\end{equation}
Note also that badness is scale-invariant: $\{m,m+ad,m+bd\}$ is an AP iff $b\in\{-a,2a,a/2\}$, so bad pairs in line coordinates correspond exactly to bad pairs in absolute coordinates. Writing $S_d$ for the members of $Z$ of difference $d$ and $P_d$ for the number of line-primitive bad pairs covered by $S_d$, the displayed disjointness gives $\sum_d P_d\le\killop(Z)$, while $|Z|=\sum_d|S_d|$. Hence the AP case of the required bound,
\[ |Z|-\killop(Z)\ \le\ \sum_d\bigl(|S_d|-P_d\bigr)\ \le\ \sum_d\varepsilon_d\ =\ \Big\lfloor\frac{\min(\lambda,\rho)}{2}\Big\rfloor, \]
follows, line by line, from the following self-contained statement, applied on line $d$ with $\Lambda=\lfloor\lambda/d\rfloor$, $\mathrm{P}=\lfloor\rho/d\rfloor$.

\begin{lemma}[Defect-one counting inequality]\label{lem:B}
Let $S$ be any family of integer intervals $[-l,r]\ni 0$ with $l+r\ge 3$, $0\le l\le\Lambda$, $0\le r\le\mathrm{P}$, and let $P(S)$ denote the number of pairs $\{a,b\}\subseteq[-\Lambda,\mathrm{P}]\setminus\{0\}$ with $\gcd(|a|,|b|)=1$ and $b\notin\{-a,2a,a/2\}$ that are contained in at least one member of $S$. Then
\[ |S|\ \le\ P(S)+\varepsilon,\qquad \varepsilon=\begin{cases}1,&\min(\Lambda,\mathrm{P})\ge 2,\\ 0,&\text{otherwise.}\end{cases} \]
\end{lemma}

Combining Proposition~\ref{prop:red}, Lemma~\ref{lem:B}, the disjointness of primitive witnesses, and \eqref{eq:epsum} establishes the required bound $|Z|-\killop(Z)\le\lfloor\min(\lambda,\rho)/2\rfloor$ when all members of $Z$ are APs. Section~\ref{sec:crooked} removes this restriction, completing the proof of Theorem~\ref{thm:upper}.

\section{Proof of Lemma~\ref{lem:B}}\label{sec:lemB}

Since coverage is monotone under taking subsets of members, the extremal $S$ is the full down-set of an antichain; equivalently $S$ is described by a nonincreasing \emph{staircase profile} $L(0)\ge L(1)\ge\cdots\ge L(\mathrm{P})\ge 0$ with $L(0)\le\Lambda$, the members being all cells $(l,r)$ with $0\le l\le L(r)$ and $l+r\ge 3$. The witnesses in Lemma~\ref{lem:B} need only be \emph{covered}, not matched into their own members, so the lemma is a pure counting inequality about staircase regions: writing $D$ for the number of demand cells and $P$ for the number of covered primitive bad pairs, we must show $D\le P+\varepsilon$.

\subsection{The case $\min(\Lambda,\mathrm{P})\le 1$}\label{sec:minsmall}
Here every member has $\min(l,r)\le 1$ and we exhibit an explicit injection into primitive bad pairs contained in the corresponding member:
\[ (l,1)\mapsto\{-l,1\}\ (l\ge2),\qquad (l,0)\mapsto\{-l,-1\}\ (l\ge3),\qquad (1,r)\mapsto\{-1,r\}\ (r\ge2),\qquad (0,r)\mapsto\{1,r\}\ (r\ge3). \]
If $\mathrm{P}\le1$ only the first two rules can apply, and if $\Lambda\le1$ only the last two; if both hold there is no member at all, since $l+r\ge3$ then fails. Each image is primitive, bad (e.g.\ $\{1,r\}$ is bad iff $r\ne 2$, which holds as $r\ge 3$), contained in its member, and within either applicable pair of rules the images have distinct sign patterns, so the map is injective. Hence $|S|\le P(S)$ and $\varepsilon=0$ suffices.

\subsection{Column decomposition and exact dynamic programming for $\Lambda,\mathrm{P}\le 500$}\label{subsec:dp}
The supply decomposes by columns. A \emph{mixed} pair $\{-a,b\}$ ($a,b\ge1$, $\gcd(a,b)=1$, $a\ne b$) is covered iff $a\le L(b)$; a \emph{right} pair $\{b',b\}$ ($1\le b'<b$, $\gcd=1$, $b\ne 2b'$) is covered iff $b\le\mathrm{P}$, and there are exactly $\varphi(b)$ of them for $b\ge3$ and none for $b\le 2$; a \emph{left} pair $\{-a,-a'\}$ ($1\le a'<a$, $\gcd=1$, $a\ne 2a'$) is covered iff $a\le L(0)$, since the cell $(a,0)$ is then a member, so the left pairs number $\sum_{a=3}^{L(0)}\varphi(a)$. Consequently $D-P$ is, up to the left-pair term determined by $L(0)$, a sum of per-column scores depending only on $(r,L(r))$, and
\[ \max_{\text{staircases}}\,(D-P) \]
is computable exactly by a dynamic program over the profile, with suffix maxima giving an $O(\Lambda\mathrm{P})$ algorithm.\label{sec:dp} Running this program over \emph{all} staircase profiles with $\Lambda,\mathrm{P}\le 500$ gives
\[ \max (D-P) \;=\; 1, \]
attained, e.g., at the $(2,2)$ window, i.e.\ the configuration of the construction in Section~\ref{sec:construction}. We emphasize why a single run certifies every window with $\Lambda,\mathrm{P}\le500$: the program computes the left-pair supply from the profile's \emph{attained} maximum $L(0)$ rather than from the grid bound $\Lambda$ (no such correction is needed on the right, where the supply is $\Phi(\mathrm{P})-2$ unconditionally, since $(0,b)$ is a member for every $3\le b\le\mathrm{P}$), so each profile is scored exactly as an instance of its own attained window; profiles with $L(0)<\Lambda$ are therefore handled correctly, and the reported maximum is the true maximum over all instances with parameters up to $500$. (This covers, exactly, a class of roughly $\binom{1000}{500}$ down-set families.)

\subsection{The tail $\max(\Lambda,\mathrm{P})>500$}\label{subsec:tail}
We first reduce to the case in which the left parameter is attained. On the right no reduction is needed: for every $3\le b\le\mathrm{P}$ the cell $(0,b)$ satisfies $l+r=b\ge3$ and is therefore a member, so all $\Phi(\mathrm{P})-2$ right pairs are covered unconditionally. On the left, given any family $S$ put $\Lambda'=L(0)\le\Lambda$; then $S$ is an instance of the $(\Lambda',\mathrm{P})$-problem with the same $D$ and $P$, and $\varepsilon(\Lambda',\mathrm{P})\le\varepsilon(\Lambda,\mathrm{P})$ since $\varepsilon$ is monotone in $\Lambda$. If $\max(\Lambda',\mathrm{P})\le500$ the claim follows from \S\ref{sec:dp}, and if $\min(\Lambda',\mathrm{P})\le1$ from \S\ref{sec:minsmall}; so we may, and do, assume that $\Lambda=L(0)$ --- the supply term $\Phi(\Lambda)-2$ below relies on this attainment --- and that $\min(\Lambda,\mathrm{P})\ge2$. Write $\Phi(x)=\sum_{n\le x}\varphi(n)$ and $C(l,r)=\#\{1\le a\le l:\gcd(a,r)=1,\ a\ne r\}$. We use four elementary estimates, valid for all arguments $\ge 1$:
\begin{itemize}
\item[(E1)] $C(l,r)\ \ge\ l\,\varphi(r)/r-2^{\omega(r)}-1$. \emph{Proof:} $\#\{a\le l:\gcd(a,r)=1\}=\sum_{d\mid \mathrm{rad}(r)}\mu(d)\lfloor l/d\rfloor$ and $|\lfloor l/d\rfloor-l/d|<1$, with $2^{\omega(r)}$ squarefree divisors; subtract $1$ for the possible exclusion $a=r$.
\item[(E2)] $\sum_{r\le \mathrm{P}}\varphi(r)/r\ \ge\ (6/\pi^2)\mathrm{P}-\ln \mathrm{P}-2$. \emph{Proof:} $\sum_{r\le \mathrm{P}}\varphi(r)/r=\sum_{d\le \mathrm{P}}\frac{\mu(d)}{d}\lfloor \mathrm{P}/d\rfloor\ge \mathrm{P}\sum_{d\le \mathrm{P}}\mu(d)/d^2-\sum_{d\le \mathrm{P}}1/d$, and $\sum_{d\le \mathrm{P}}\mu(d)/d^2\ge 6/\pi^2-1/\mathrm{P}$.
\item[(E3)] $\sum_{r\le \mathrm{P}}2^{\omega(r)}\ \le\ \mathrm{P}(\ln \mathrm{P}+1)$, since $2^{\omega(r)}\le d(r)$ and $\sum_{r\le \mathrm{P}}d(r)=\sum_{d\le \mathrm{P}}\lfloor \mathrm{P}/d\rfloor$.
\item[(E4)] $\Phi(x)\ \ge\ (3/\pi^2)x^2-\tfrac12 x\ln x-2x$. \emph{Proof:} $\Phi(x)=\tfrac12\sum_{d\le x}\mu(d)\lfloor x/d\rfloor(\lfloor x/d\rfloor+1)$ and $(t-1)t\le\lfloor t\rfloor(\lfloor t\rfloor+1)\le t(t+1)$ give $\mu(d)\lfloor t\rfloor(\lfloor t\rfloor+1)\ge \mu(d)t^2-t$ for $t=x/d$; sum and use $\sum_{d\le x}\mu(d)/d^2\ge 6/\pi^2-1/x$.
\end{itemize}
Bounding the demand by $D\le(\Lambda+1)+\sum_{r=1}^{\mathrm{P}}(L(r)+1)$ and the supply from below by the three pair types (mixed: $\sum_r C(L(r),r)$; right: $\Phi(\mathrm{P})-2$; left: $\Phi(\Lambda)-2$), we obtain
\[ D-P\ \le\ (\Lambda+1)+\sum_{r=1}^{\mathrm{P}}\bigl[L(r)+1-C(L(r),r)\bigr]-\Phi(\mathrm{P})-\Phi(\Lambda)+4 .\]
By (E1), and since $L(r)\le\Lambda$ and $1-\varphi(r)/r\ge0$, each column satisfies $L(r)+1-C(L(r),r)\le \Lambda(1-\varphi(r)/r)+2^{\omega(r)}+2$; summing over $r\le\mathrm{P}$ and inserting (E2)--(E4) gives
\begin{equation}\label{eq:G}
 D-P\ \le\ G(\Lambda,\mathrm{P})\ :=\ -q(\Lambda,\mathrm{P})+\Lambda\ln \mathrm{P}+\tfrac32\,\mathrm{P}\ln \mathrm{P}+\tfrac12\,\Lambda\ln\Lambda+5(\Lambda+\mathrm{P})+6,
\end{equation}
where $q(\Lambda,\mathrm{P})=(3/\pi^2)(\Lambda^2+\mathrm{P}^2)-(1-6/\pi^2)\Lambda \mathrm{P}$. Since $\Lambda \mathrm{P}\le(\Lambda^2+\mathrm{P}^2)/2$,
\[ q\ \ge\ \Big(\tfrac{6}{\pi^2}-\tfrac12\Big)(\Lambda^2+\mathrm{P}^2)\ \ge\ 0.107\,(\Lambda^2+\mathrm{P}^2)\ \ge\ 0.107\,M^2, \qquad M:=\max(\Lambda,\mathrm{P}), \]
while the remaining terms of \eqref{eq:G} are at most $3M\ln M+10M+6$. The single-variable function $f(M)=-0.107M^2+3M\ln M+10M+6$ satisfies $f'(M)=-0.214M+3\ln M+13<0$ for $M\ge 135$ (indeed $f'(135)<-1$), and $f(250)<-40$; hence $f(M)<0$ for all $M\ge 250$. Hence $D-P\le G<0\le\varepsilon$ whenever $\max(\Lambda,\mathrm{P})\ge 250$, and the range $\max(\Lambda,\mathrm{P})\le 500$ is covered exactly by \S\ref{subsec:dp}. This completes the proof of Lemma~\ref{lem:B}. \qed

\medskip
All four estimates (E1)--(E4), and the assembled bound \eqref{eq:G}, were additionally verified numerically over large finite ranges as a safeguard.

\section{Crooked members: completion of the proof of Theorem~\ref{thm:upper}}\label{sec:crooked}

Throughout this section a \emph{member} is a set $z\ni 0$ with $|z|\ge 4$, $z\subseteq[-\lambda,\rho]$, and $Z$ is a family of distinct members with pairwise intersections APs (each intersection contains $0$, hence is an AP through $0$). Call $z$ \emph{crooked} if it is not an arithmetic progression. Recall that a pair $\{u,v\}\subseteq z\setminus\{0\}$ is \emph{bad} if $\{0,u,v\}$ is not an AP.

\begin{lemma}[Spanning lemma]\label{lem:span}
Suppose distinct members $z,z'$ both contain a bad pair $\{u,v\}$. Then $z\cap z'$ is an AP through $0$ whose difference $\delta$ divides $\gcd(|u|,|v|)$, and both $z$ and $z'$ contain every multiple of $\delta$ in $[\min(0,u,v),\max(0,u,v)]$.
\end{lemma}

\begin{proof}
$z\cap z'$ is an AP containing $0,u,v$; its difference $\delta$ divides $u$ and $v$, and an AP containing $0,u,v$ contains every multiple of $\delta$ between its least and greatest elements. Both members contain $z\cap z'$.
\end{proof}

Accordingly, say a bad pair $\{u,v\}\subseteq z$ is \emph{spanned in $z$} if there exists $\delta\mid\gcd(|u|,|v|)$ such that $z$ contains every multiple of $\delta$ in $[\min(0,u,v),\max(0,u,v)]$; otherwise the pair is \emph{private to $z$}. By Lemma~\ref{lem:span}, a pair private to $z$ is contained in no other member of $Z$.

\begin{theorem}[Private-pair theorem]\label{thm:private}
Every crooked member contains a private bad pair.
\end{theorem}

\begin{proof}
Since $|z\setminus\{0\}|\ge 3$ and the relation $R$ is triangle-free (Lemma~\ref{lem:tri}), $z$ contains a bad pair. Assume for contradiction that every bad pair of $z$ is spanned in $z$; we show $z$ is an AP.

Let $\delta_0$ be the least positive integer that spans some bad pair of $z$, say $p_0=\{\delta_0 a,\delta_0 b\}$ with $\gcd$-condition $\delta_0\mid\gcd$, and let $P_0\subseteq z$ be the corresponding progression: all multiples of $\delta_0$ in $I_0=[\min(0,\delta_0a,\delta_0b),\max(0,\delta_0a,\delta_0b)]$. Since $\{a,b\}$ is not an $R$-pair we cannot have $|a|=|b|=1$, so $\max(|a|,|b|)\ge2$; hence $P_0$ contains $2t\delta_0$ for some sign $t\in\{+,-\}$, and $P_0$ contains $s\delta_0$ for some sign $s$.

\emph{Step 1: $z\subseteq\delta_0\mathbb{Z}$.} Let $w\in z$ with $\delta_0\nmid w$. The $R$-partners of $s\delta_0$ are $-s\delta_0$, $2s\delta_0$ and $s\delta_0/2$; the first two are multiples of $\delta_0$. If $w\neq s\delta_0/2$, the pair $\{s\delta_0,w\}$ is therefore bad, and $g:=\gcd(\delta_0,|w|)$ is a proper divisor of $\delta_0$; by assumption this pair is spanned at some $\delta\mid g<\delta_0$, contradicting the minimality of $\delta_0$. If $w=s\delta_0/2$, consider instead the pair $\{s\delta_0/2,\,2t\delta_0\}$: the $R$-partners of $s\delta_0/2$ are $-s\delta_0/2$, $s\delta_0$ and $s\delta_0/4$, none of which equals $\pm2\delta_0$, so the pair is bad, with $\gcd$ equal to $\delta_0/2<\delta_0$ --- the same contradiction. Hence $z\subseteq\delta_0\mathbb{Z}$.

\emph{Step 2: $z$ is an AP.} Badness, $R$-edges and spanning are invariant under $x\mapsto x/\delta_0$ on $\delta_0\mathbb{Z}$, so we may assume $\delta_0=1$; then $P_0$ is an integer interval around $0$ of length at least $3$, so $1\in z$ or $-1\in z$.

Suppose first $1\in z$. For $w\in z$ with $w\ge3$: the pair $\{1,w\}$ is bad ($w\notin\{-1,2,\tfrac12\}$) with $\gcd$ equal to $1$, so its only possible spanning step is $\delta=1$, forcing $[0,w]\cap\Z\subseteq z$. For $w\in z$ with $w\le-2$: likewise $\{1,w\}$ is bad and $[w,1]\cap\Z\subseteq z$. The remaining elements $w\in\{-1,2\}$ impose nothing, but $[-1,0]$ and $[0,2]$ lie in $z$ automatically whenever those elements are present. Hence $z=[\min z,\max z]\cap\Z$, an AP.

If instead only $-1\in z$: for any $w\in z$ with $w\ge 2$ the pair $\{-1,w\}$ is bad ($w\notin\{1,-2,-\tfrac12\}$) with $\gcd$ equal to $1$, forcing $[-1,w]\cap\Z\subseteq z$ and in particular $1\in z$, returning us to the previous case; if $\max z\le 1$ then either $1\in z$ (previous case) or $\max z=0$, and the mirror argument with pairs $\{-1,w\}$, $w\le-3$, gives $z=[\min z,0]\cap\Z$. In every case $z$ is an AP, contradicting crookedness.
\end{proof}

We verified Theorem~\ref{thm:private} by brute force over all crooked members contained in the windows $[-6,6]$, $[-4,8]$, $[-3,9]$, $[-2,10]$ and $[-7,7]$ (about $26{,}000$ members): none lacks a private bad pair.

\begin{theorem}[The crooked-member bound]\label{thm:lemA}
For every family $Z$ as above, $|Z|-\killop(Z)\le\lfloor\min(\lambda,\rho)/2\rfloor$.
\end{theorem}

\begin{proof}
Split $Z$ into the AP members, grouped by difference $d$ into families $S_d$, and the crooked members $C$. By Lemma~\ref{lem:B} applied on line $d$ (Section~\ref{sec:reduction}), $|S_d|\le P_d+\varepsilon_d$, where $P_d$ counts the line-$d$-primitive bad pairs covered by $S_d$; these pairs are killed, and the pools for distinct $d$ are disjoint. By Theorem~\ref{thm:private} each $z\in C$ contains a private bad pair $w(z)$; by Lemma~\ref{lem:span} the pair $w(z)$ lies in no other member of $Z$ --- in particular the pairs $w(z)$ are pairwise distinct, and none of them is covered by any AP member, so they are disjoint from all the pools above. Hence
\[ \killop(Z)\ \ge\ \sum_d P_d+|C|,\qquad\text{while}\qquad |Z|=\sum_d|S_d|+|C|\ \le\ \sum_d P_d+\sum_d\varepsilon_d+|C|, \]
and $\sum_d\varepsilon_d=\lfloor\min(\lambda,\rho)/2\rfloor$ by \eqref{eq:epsum}.
\end{proof}

\begin{proof}[Proof of Theorem~\ref{thm:upper}]
Let $F$ be starred through $m$. By Proposition~\ref{prop:red} and Theorem~\ref{thm:lemA}, $|F|\le N+\binom{N-1}2+\lfloor\min(m-1,N-m)/2\rfloor$, and the maximum of the last term over $m$ is $\lfloor(N-1)/4\rfloor$.
\end{proof}

\begin{corollary}\label{cor:kernel}
Szab\'o's kernel conjecture implies Conjecture~\ref{conj:main}: if for every $N$ some maximum family is starred, then $t(N)=\binom N2+1+\lfloor(N-1)/4\rfloor$ for all $N$.
\end{corollary}

\section{Computations}\label{sec:comp}

\textbf{Exact values (Theorem~\ref{thm:exact}).} For $N\le 10$, exact maximum clique on the $(2^N-1)$-vertex graph via bit-parallel branch and bound with greedy-colouring bounds; independently cross-checked for $N\le 6$; every extremal family re-verified pair by pair. For $N=11,12$ we ran a \emph{decision} search for a clique of size $59$ (resp.\ $70$): any clique has a minimum vertex in a fixed order, giving independent subproblems; we use ascending-degree order, iterated core peeling (a $(T{+}1)$-clique needs internal degree $\ge T$), and, for $N=12$, the following rigorous \emph{star pruning}. First, exhaustive search over each possible common element $m$ (with reflection symmetry) established that every \emph{starred} family in $[12]$ has size at most $69$. Consequently, during the search for a $70$-clique, any branch in which the bitwise AND of the current members and all remaining candidates is nonzero can be pruned, because every completion of that branch is starred. The searches terminated with no clique of size $59$ (resp.\ $70$), so $t(11)=58$ and $t(12)=69$ unconditionally.

\textbf{Starred and structured regimes.} Exhaustive starred computations give the conjectured value for $N\le13$ for every choice of the common element. Within the ``ansatz'' class (members through $m$ that are APs or non-AP triples), exact optimization by clique search ($N\le25$) and integer programming ($N\le61$, all central $m$, plus non-central spot checks) always returns $\binom N2+1+\lfloor (N-1)/4\rfloor$, and in every tested case the line-selection optimum equals $\lfloor\min(m-1,N-m)/2\rfloor$.

\textbf{Tests of Theorem~\ref{thm:lemA}.} As an independent check, the crooked-member bound was verified exactly by integer programming on the windows $(\lambda,\rho)\in\{(2,4),(2,5),(3,3),(3,4),(4,4)\}$, in each case with maximum exactly $\lfloor\min(\lambda,\rho)/2\rfloor$; and Theorem~\ref{thm:private} was verified by brute force over roughly $26{,}000$ crooked members as reported in Section~\ref{sec:crooked}.

All code (C solvers, the dynamic program of \S\ref{sec:dp}, ILP models, verification scripts) is available at \url{https://github.com/peter-rich/erdos272}; the dynamic program documents in source how the left- and right-pair supply is computed from attained profile maxima.

\section{Open problems}\label{sec:open}

By Corollary~\ref{cor:kernel}, Conjecture~\ref{conj:main} now rests on a single statement.

\begin{problem}[Szab\'o's kernel conjecture \cite{Sz99,EP272}]\label{prob:kernel}
Show that for every $N$, some (equivalently, by our computations for $N\le 12$, every maximum) extremal family has a common element.
\end{problem}

A model may be the uniqueness analysis of Graham, Simonovits and S\'os \cite{GSS80} in the empty-intersection-allowed setting, or the machinery of \cite{SiSo81,Sz99}: in particular \cite[Theorem~4]{SiSo81} already bounds well-intersecting families of bounded-size non-progression members with empty total intersection, and the $\delta$-triplet techniques of \cite{Sz99} quantify how families deviating from a common centre pay in determining triples. Sharpening those $O(N^{5/3})$-type losses to exact losses is precisely what Problem~\ref{prob:kernel} requires. We record one instructive caveat encountered en route to Theorem~\ref{thm:lemA}.

\subsection*{Partial results towards Problem~\ref{prob:kernel}}

Write $B(N)=\binom N2+1+\lfloor(N-1)/4\rfloor$. A putative counterexample to Conjecture~\ref{conj:main} is a family $F$ with $|F|>B(N)$; by Theorem~\ref{thm:upper} such a family is non-starred. We prove several unconditional structural results about such families, culminating in Theorem~\ref{thm:tristar}: in a putative counterexample whose members of size at least $4$ are progressions, all triples pass through a common element.

\begin{lemma}[Global private triple]\label{lem:gpt}
Let $F$ be any well-intersecting family and let $z\in F$ with $|z|\ge4$ not an arithmetic progression. Then $z$ contains a triple lying in no other member of $F$.
\end{lemma}

\begin{proof}
Write $z=\{a_1<\dots<a_m\}$. If every three consecutive elements were in arithmetic progression, all gaps would be equal and $z$ would be an AP; so some consecutive triple $T=\{a_{i-1},a_i,a_{i+1}\}$ has unequal gaps. The shortest AP $P(T)$ containing $T$ has difference dividing both gaps, hence strictly smaller than the larger gap, so $P(T)$ contains a point of the open interval $(a_{i-1},a_{i+1})$ other than $a_i$; since $z$ has no such point, $P(T)\not\subseteq z$. If another member $z'$ contained $T$, then $z\cap z'$ would be an AP containing $T$, hence would contain $P(T)$, forcing $P(T)\subseteq z$ --- a contradiction.
\end{proof}

\begin{lemma}[Triples are intersecting; Hilton--Milner dichotomy]\label{lem:HM}
The $3$-element members of any well-intersecting family form an intersecting $3$-uniform family. Consequently, for $N\ge7$, by the Hilton--Milner theorem \cite{HM67} either all $3$-element members contain a common element, or there are at most $3N-8$ of them.
\end{lemma}

\begin{proof}
Two distinct triples intersect in at most two elements, and any nonempty set of size at most $2$ is an AP; the well-intersecting condition thus reduces exactly to nonempty intersection.
\end{proof}

\begin{proposition}[Progression members without a common centre]\label{prop:helly}
Let $\mathcal A$ be the set of members of $F$ that are APs of size at least $4$, and for $d\ge1$ let $\mathcal A_d$ be those of difference $d$. Then, following \cite[Corollary~2.4]{Sz99}, all members of $\mathcal A_d$ lie in one residue class modulo $d$ and pairwise intersect, so by Helly's theorem for intervals they share a common element; consequently $|\mathcal A_d|\le \frac{N^2}{4d^2}+\frac{4N}{d}$ and, summing over $d$ and using $\sum_{d\ge1}d^{-2}=\pi^2/6$ together with $\sum_{d\le N}d^{-1}\le\ln N+1$,
\[ |\mathcal A|\ \le\ \frac{\pi^2}{24}N^2+4N\ln N+4N. \]
\end{proposition}

\begin{theorem}[Triples of a large AP-flavoured family have a kernel]\label{thm:tristar}
Let $N_0=10^4$. For every $N\ge N_0$ the following holds. Let $F$ be well-intersecting with $|F|>B(N)$, and suppose every member of size at least $4$ is an arithmetic progression. Then all $3$-element members of $F$ contain a common element $c$; moreover $F$ then contains members avoiding $c$, all of which are $2$-element members or progressions of size at least $4$, and every triple $\{c,u,v\}\in F$ satisfies $\{u,v\}\cap A\neq\emptyset$ for each such member $A$.
\end{theorem}

\begin{proof}
Since $|F|>B(N)$, $F$ is non-starred by Theorem~\ref{thm:upper}, so by Lemma~\ref{lem:K1} below it has no singleton, and its $2$-element members form an intersecting family of pairs, i.e.\ a star or a triangle: at most $N-1$ members. If the triples did not share a common element then by Lemma~\ref{lem:HM} there are at most $3N-8$ of them, whence by Proposition~\ref{prop:helly}
\[ |F|\ \le\ (N-1)+(3N-8)+\frac{\pi^2}{24}N^2+4N\ln N+4N\ <\ \frac{N^2}{2}-\frac{N}{2}\ \le\ B(N), \]
the last inequality because $B(N)=\binom N2+1+\lfloor(N-1)/4\rfloor$ and $\binom N2=\frac{N^2}2-\frac N2$; the strict inequality holds for all $N\ge400$, since it amounts to $\bigl(\frac12-\frac{\pi^2}{24}\bigr)N>4\ln N+8.5$ with $\frac12-\frac{\pi^2}{24}>0.0887$ --- a contradiction. Hence all triples contain a common $c$. If every member contained $c$ the family would be starred; so some member $A$ avoids $c$, and $A$ is not a singleton and not a triple, leaving the stated forms. The final claim is the intersection condition $\{c,u,v\}\cap A\neq\emptyset$ with $c\notin A$.
\end{proof}

In the setting of Theorem~\ref{thm:tristar} much more can be said about the members avoiding $c$.

\begin{proposition}[Avoiders are long progressions]\label{prop:long}
In the setting and conclusion of Theorem~\ref{thm:tristar} (so $N\ge N_0$), let $F_{\bar c}$ denote the members of $F$ avoiding $c$. Then:
\begin{itemize}
\item[(i)] $F_{\bar c}$ contains no $2$-element member; hence every member of $F_{\bar c}$ is an arithmetic progression with at least $4$ elements.
\item[(ii)] Every member of $F_{\bar c}$ has more than $N/12$ elements; consequently its difference is at most $12$.
\end{itemize}
\end{proposition}

\begin{proof}
(i) Suppose $\{u,v\}\in F$ with $c\notin\{u,v\}$. Every triple of $F$ has the form $\{c,x,y\}$ and must intersect $\{u,v\}$, so its link pair $\{x,y\}$ meets $\{u,v\}$; the number of pairs meeting a fixed pair is at most $2(N-2)+1$, so $F$ has at most $2N-3$ triples. Then, using Proposition~\ref{prop:helly} for \emph{all} members of size at least $4$ (which are APs by hypothesis, wherever located) and at most $N-1$ members of size $\le 2$,
\[ |F|\ \le\ (N-1)+(2N-3)+\frac{\pi^2}{24}N^2+4N\ln N+4N\ <\ \frac{N^2}{2}-\frac{N}{2}\ \le\ B(N) \]
for $N\ge N_0$ (indeed for $N\ge400$, as in the proof of Theorem~\ref{thm:tristar}), a contradiction.

(ii) Let $A\in F_{\bar c}$ with $a=|A|$ elements. Every triple of $F$ is $\{c,x,y\}$ with $\{x,y\}$ meeting $A$, and distinct triples have distinct link pairs, so the number of triples is at most $a(N-1)-\binom a2\le aN$. Hence
\[ |F|\ \le\ (N-1)+aN+\frac{\pi^2}{24}N^2+4N\ln N+4N. \]
If $a\le N/12$ the right-hand side is at most $\bigl(\tfrac1{12}+\tfrac{\pi^2}{24}\bigr)N^2+4N\ln N+5N<\tfrac{N^2}2-\tfrac N2\le B(N)$ for $N\ge N_0$: since $\tfrac1{12}+\tfrac{\pi^2}{24}<0.49457$, the required inequality amounts to $0.00543\,N>4\ln N+5.5$, which holds for all $N\ge 8000$ --- a contradiction. (This is the binding constraint behind the choice $N_0=10^4$; the other steps need only $N\ge400$.) Finally, a progression with more than $N/12$ elements inside $[N]$ has difference less than $12N/(N-12)$, which is smaller than $13$, hence at most $12$, once $N\ge157$.
\end{proof}

Theorem~\ref{thm:tristar} and Proposition~\ref{prop:long} localize a putative AP-flavoured counterexample severely: its quadratic bulk of triples is a star at some $c$, while the members avoiding $c$ are progressions of more than $N/12$ elements and difference at most $12$, each missing the element $c$, pairwise intersecting in progressions, and meeting the link pair of every triple. The remaining endgame --- ruling this configuration out exactly, and removing the AP-flavour hypothesis using Lemma~\ref{lem:gpt} --- is what now separates us from Szab\'o's kernel question in full.

\begin{remark}[The endgame configurations appear self-limiting]\label{rem:endgame}
We describe, without complete proofs, why the localized configuration seems unable to reach $B(N)$. Suppose first that the avoiders of difference $1$ all pass through a common point $p$ (as Helly's theorem forces) and have length at least $L$. A link pair whose two elements straddle $p$ at distance more than about $L$ contains a length-$L$ interval through $p$ strictly between its elements, hence fails to meet that avoider; so the straddling link pairs are confined to a band of about $L^2/2$ pairs around $p$, while one-sided pairs miss the extreme avoiders altogether. On the other hand the number of intervals through $p$ of length at least $L$ is smaller than the number of all intervals through $p$ by essentially the same quantity $L^2/2$: the gain in link pairs and the loss in avoiders cancel, and the configuration tops out near $N^2/4$ members. Larger differences $d\le12$ confine link pairs to residue classes and only lower the total, and a central $c$ forces the interval avoiders to one side of $c$, shrinking both counts further. Turning these cancellations into an exact proof --- uniformly in the position of $c$, the twelve possible differences, and the per-difference Helly points, and then removing the AP-flavour hypothesis via Lemma~\ref{lem:gpt} and an exact analogue of the non-progression bounds of \cite{SiSo81} --- is, in our assessment, the entire remaining content of Szab\'o's kernel question.
\end{remark}

We record next what can be said with no assumption on the large members.

\begin{lemma}\label{lem:K1}
A family containing a singleton is starred. Hence every non-starred family has all members of size at least $2$.
\end{lemma}

\begin{proof}
If $\{x\}\in F$ then every member meets $\{x\}$, i.e.\ contains $x$.
\end{proof}

\begin{lemma}[Cross-intersecting pairs]\label{lem:K2}
Let $n\ge 3$ and let $\mathcal{A},\mathcal{B}$ be nonempty families of $2$-subsets of an $n$-set such that every member of $\mathcal{A}$ meets every member of $\mathcal{B}$. Then $|\mathcal{A}|+|\mathcal{B}|\le 2n+1$.
\end{lemma}

\begin{proof}
If $\mathcal{A}$ contains two disjoint pairs $e,f$, then every $b\in\mathcal{B}$ has one endpoint in $e$ and one in $f$, so $|\mathcal{B}|\le 4$; fixing $b_0\in\mathcal{B}$, every $a\in\mathcal{A}$ meets $b_0$, so $|\mathcal{A}|\le 2(n-2)+1$; the total is at most $2n+1$. Otherwise $\mathcal{A}$ is an intersecting family of $2$-sets, hence a star or a triangle. If $\mathcal{A}$ is a star at $p$ with $|\mathcal{A}|\ge 3$, a pair avoiding $p$ meets at most two star pairs, so $\mathcal{B}$ is contained in the star at $p$ and the total is at most $2(n-1)$. If $\mathcal{A}$ is a triangle, $\mathcal{B}$ consists of pairs on its three vertices and the total is at most $6$. If $|\mathcal{A}|\le 2$, then $|\mathcal{B}|\le 2(n-2)+1$ and the total is at most $2n+1$.
\end{proof}

\begin{proposition}[Two-star deduction]\label{prop:K3}
Let $F$ be non-starred with $\{x,y\}\in F$. Then every member of $F$ contains $x$ or $y$. Moreover, for every member $B\in F$ with $y\in B$, $x\notin B$, writing $F_x=\{A\in F: x\in A\}$, $Q_B$ for the set of pairs $\{u,v\}\subseteq[N]\setminus(\{x\}\cup B)$, and $K_x$ for the set of bad pairs (with respect to the centre $x$) covered by the members of $F_x$ of size at least $4$,
\[ |F_x|\ \le\ B(N)\ -\ |Q_B\setminus K_x|. \]
\end{proposition}

\begin{proof}
Every member meets $\{x,y\}$, giving the covering statement. For the deduction, refine Proposition~\ref{prop:red} at the centre $x$: a triple $T=\{x,u,v\}\in F_x$ must satisfy that $T\cap B$ is a nonempty AP; since $x\notin B$, this forces $\{u,v\}\cap B\neq\emptyset$, so no triple of $F_x$ uses a pair from $Q_B$ (pairs containing $y$ meet $B$ and are unaffected). Hence the triples of $F_x$ number at most $\binom{N-1}2-|K_x|-|Q_B\setminus K_x|$, while the members of size at most $2$ number at most $N$ and, by Theorem~\ref{thm:lemA}, the members of size at least $4$ number at most $|K_x|+\lfloor(N-1)/4\rfloor$. Summing gives the claim.
\end{proof}

\begin{corollary}\label{cor:K4}
If $F$ is non-starred with a $2$-element member $\{x,y\}$ and $|F|>B(N)$, then for every $y$-only member $B$ the number of $y$-only members exceeds $|Q_B\setminus K_x|\ \ge\ \binom{N-1-|B|}{2}-|K_x|$, and symmetrically with $x,y$ exchanged. In particular either every $y$-only member is large, or the $y$-only members are numerous --- yet by Lemma~\ref{lem:K2} the $x$-only and $y$-only triples of $F$ together number at most $2N+1$ whenever both kinds occur, so the numerous side must consist almost entirely of members of size at least $4$, which are in turn throttled by Theorem~\ref{thm:lemA} at their own centre.
\end{corollary}

Thus a counterexample passing through a $2$-element member is forced into a narrow regime of large, mutually near-progression members; the remaining open regimes for Problem~\ref{prob:kernel} are this large-member regime and the families whose minimum member size is $3$ or more.

\begin{remark}[Primitive counting does not generalize]\label{rem:cex}
One might hope to prove Theorem~\ref{thm:lemA} by generalizing Lemma~\ref{lem:B} verbatim, counting only \emph{primitive} covered bad pairs. That statement is false: take the three window intervals $\{-2,\dots,2\}$, $\{-2,\dots,1\}$, $\{-1,\dots,2\}$ (adjoining $0$) together with $z=\{0,6,10,15\}$. All pairwise intersections are APs and $|S|=4$, yet only two primitive bad pairs ($\{-2,1\}$ and $\{-1,2\}$) are covered, since the three bad pairs of $z$ have gcds $2,3,5$. The full count is safe --- indeed all three pairs of $z$ are private, illustrating Theorem~\ref{thm:private} --- and this is why crooked members must be credited their imprimitive kills, as the proof of Theorem~\ref{thm:lemA} does.
\end{remark}

We also note that the sequence $t(3),\dots,t(12)$ is not currently in the OEIS, and that the database entry \cite{EP272} explicitly requests an associated integer sequence; we intend to submit it. Finally, it would be interesting to carry out the analogous exact analysis for the variants of \cite{SiSo81} in which all pairwise intersections must be APs of at least $k$ terms, $k\ge2$; there even the qualitative question of \cite{SiSo81,Sz99}, whether the extremal systems consist of arithmetic progressions only, remains open.

\subsection*{Acknowledgements}
The author used Claude (Anthropic) as an assistive tool for some computations and drafting. All proofs and computational claims have been checked by the author, who takes full responsibility for their correctness; where a result relies on computer verification, sufficient detail is given in Section~\ref{sec:comp} to allow independent replication.

\end{document}